\documentclass{article}
\usepackage[utf8]{inputenc}

\usepackage{fullpage}
\usepackage{times}
\usepackage{url}
\usepackage{natbib}
\usepackage{amsthm,amsfonts,amsmath,amssymb,epsfig,color,float,graphicx,verbatim}
\usepackage{enumitem}
\usepackage{algorithm}
\usepackage{algpseudocode}

\usepackage{wrapfig}
\usepackage{hyperref}
\usepackage{multirow}
\usepackage{tabularx}
\usepackage{prettyref}

\hypersetup{
    colorlinks=true,
    linkcolor=blue,
    filecolor=blue,      
    urlcolor=blue,
    citecolor= blue
    }

\newtheorem{theorem}{Theorem}
\newtheorem{proposition}{Proposition}
\newtheorem{lemma}{Lemma}

\newtheorem{definition}{Definition}

\newtheorem{claim}{Claim}

\newcommand{\e}{\mathrm{e}}
\newcommand{\reals}{\mathbb{R}}
\newcommand{\conv}{\mathrm{conv}}
\newcommand{\sign}{\mathrm{sign}}
\newcommand{\dist}{\mathrm{dist}}
\newcommand{\prog}{\mathrm{prog}}
\newcommand{\GD}{\mathrm{GD}}
\newcommand{\INGD}{\mathrm{INGD}}
\newcommand{\E}{\mathbb{E}}
\newcommand{\NN}{\mathbb{N}}

\newcommand{\Acal}{\mathcal{A}}
\newcommand{\Dcal}{\mathcal{D}}
\newcommand{\Fcal}{\mathcal{F}}

\newcommand{\norm}[1]{\left\|#1\right\|}

\newcommand{\inner}[1]{\left\langle#1\right\rangle}

\newcommand{\secref}[1]{Section~\ref{#1}}

\newcommand{\figref}[1]{Fig.~\ref{#1}}
\renewcommand{\eqref}[1]{Eq.~(\ref{#1})}
\newcommand{\lemref}[1]{Lemma~\ref{#1}}
\newcommand{\claimref}[1]{Claim~\ref{#1}}

\newcommand{\thmref}[1]{Theorem~\ref{#1}}
\newcommand{\propref}[1]{Proposition~\ref{#1}}

\title{On the Complexity of Finding Small Subgradients in Nonsmooth Optimization}
\author{Guy Kornowski  \qquad Ohad Shamir\\
  Weizmann Institute of Science \\
}

\date{}

\begin{document}

\maketitle

\begin{abstract}
    We study the oracle complexity of producing $(\delta,\epsilon)$-stationary points of Lipschitz functions, in the sense proposed by \citet{zhang2020complexity}. While there exist dimension-free randomized algorithms for producing such points within $\widetilde{O}(1/\delta\epsilon^3)$ first-order oracle calls, we show that no dimension-free rate can be achieved by a deterministic algorithm. On the other hand, we point out that this rate can be derandomized for smooth functions with merely a logarithmic dependence on the smoothness parameter. Moreover, we establish several lower bounds for this task which hold for any randomized algorithm, with or without convexity.
    Finally, we show how the convergence rate of finding $(\delta,\epsilon)$-stationary points can be improved in case the function is convex, a setting which we motivate by proving that in general no finite time algorithm can produce points with small subgradients even for convex functions.
\end{abstract}

\section{Introduction}

We consider the problem of optimizing Lipschitz continuous functions $f:\reals^d\to\reals$ using a first-order algorithm, which utilizes values and derivatives of the function at various points. Though problems of this type are ubiquitous throughout modern machine learning, they are well known to be impossible to solve at a dimension-free rate without further assumptions such as convexity or smoothness.
For example, it is generally impossible to obtain local minima or approximate-stationary points of $f$ \citep{nemirovskiyudin1983,zhang2020complexity}, nor is it even possible to get close to such points within any finite time independent of $d$ \citep{kornowski2021oracle}.
Seeking to design reachable optimization goals for this large function class, \citet{zhang2020complexity} proposed the relaxed notion of a $(\delta,\epsilon)$-stationary point. Simply put, these are points for which there exists a convex combination of gradients in a $\delta$-neighborhood whose norm is less than $\epsilon$ (see \secref{sec: preliminaries} for a formal definition).
Their main contribution is a randomized algorithm, $\INGD$, which produces $(\delta,\epsilon)$-stationary points of any Lipschitz, bounded from below function within $\widetilde{O}(1/\delta\epsilon^3)$ iterations. Although $\INGD$ initially required access to a slightly nonstandard oracle, follow-up works modified the algorithm such that it will rely on access to a standard first-order oracle at differentiable points \citep{davis2021gradient,tian2022finite}. All of these algorithms are randomized and share the same oracle complexity of $\widetilde{O}(1/\delta\epsilon^3)$.

In this work, we aim towards a better understanding of the oracle complexity of producing $(\delta,\epsilon)$-stationary points in various settings.
First, we examine whether the algorithms mentioned above can be derandomized. 
Namely, what rate can be achieved by a deterministic algorithm that produces $(\delta,\epsilon)$-stationary points?
We solve this question by showing a strong lower bound, proving that deterministic algorithms cannot achieve \emph{any} dimension-free rate.
On the other hand, we point out that if the objective function is even slightly smooth, it is possible to provide a deterministic algorithm whose rate is $\widetilde{O}(1/\delta\epsilon^3)$ with only a logarithmic dependence on the smoothness parameter.

Next, we take the first step towards providing lower bounds which hold for any (possibly randomized) first-order method.
In terms of $\epsilon$ dependence, we provide a lower bound of $\Omega(1/\epsilon^2)$ which holds even for convex functions. In terms of $\delta$ dependence, we provide an $\Omega(\log(1/\delta))$ lower bound for nonconvex functions.
Though these results are clearly not tight, there are currently no randomized lower bounds whatsoever for this setting in terms of neither $\epsilon$ nor $\delta$. Noticeably, \citet{zhang2020complexity} do prove an $\Omega(1/\delta)$ lower bound, though their proof covers only deterministic algorithms which have access to the derivative alone (without function value).
As we conjecture the lower bounds are loose, especially with respect to $\delta$, providing tighter and unified lower bounds - in terms of $\delta$ \emph{and} $\epsilon$ - remains an intriguing open question.

Finally, we also examine the complexity of finding $(\delta,\epsilon)$-stationary of \emph{convex} functions. Although the vast majority of convex optimization literature  deals with minimizing the function value \citep{nesterov2018lectures}, some applications require the design of algorithms that produce points with small subgradients (see \citep{nesterov2012make,allen2018make} and discussions therein). Moreover, while for \emph{smooth} convex function the picture is relatively well understood in terms of complexity upper and lower bounds \citep{carmon2021lowerII}, for nonsmooth convex optimization much less is known in terms of minimizing subgradient norm.
We start by proving that even for convex yet nonsmooth functions, it is generally impossible to produce $\epsilon$-stationary points - namely, with some subgradient of norm at most $\epsilon$ (corresponding to $\delta\to0$). In fact, we prove an explicit $\delta$-dependent lower bound for convex functions, hinting that it might be beneficial to consider the relaxation of $(\delta,\epsilon)$-stationarity in this setting as well.
Furthermore, we show that the complexity of producing a $(\delta,\epsilon)$-stationary point of a convex function improves (when compared to the nonconvex case) to $\widetilde{O}(1/\delta^{2/3}\epsilon^2)$ iterations using a randomized algorithm, or (again) even a deterministic algorithm if the function is slightly smooth. These algorithms have optimal $\epsilon$-dependence as they match our aforementioned lower bound.

Overall, we hope this work will motivate further understanding of complexity guarantees for nonsmooth optimization as a whole, and of producing $(\delta,\epsilon)$-stationary points in particular.

This paper is structured as follows. In \secref{sec: preliminaries} we formally introduce notation and terminology which we use throughout the paper. In \secref{sec: deterministic} we present our results for deterministic algorithms, while in \secref{sec: rand lower} we provide our randomized lower bounds. In \secref{sec: convex stat} we turn to analyze the case of convex functions. We conclude in \secref{sec: discussion} with a discussion of our results and future directions.
\secref{sec: proofs} contains the full proofs of our results.

\section{Preliminaries} \label{sec: preliminaries}

\paragraph{Notation.}
We denote by $\reals^d$ the $d$-dimensional Euclidean space, any by $\inner{\cdot,\cdot},\norm{\,\cdot\,}$ the standard inner product and its associated Euclidean norm, respectively. Given a point $x\in\reals^d$ we denote by $x_i$ its $i$'th coordinate, and for any subset $A\subset\reals^d$ we denote their distance $\dist(x,A):=\inf_{y\in A}\norm{x-y}$.
We let $\e_i\in\reals^d$ be the $i$'th standard basis vector, $\NN=\{1,2,\dots\}$ be the natural numbers starting from 1, and $[N]=\{1,2,\dots,N\}$ be the natural numbers up to $N\in\NN$. We denote by $B_{r}(x)$ the open Euclidean ball around $x\in\reals^d$ of radius $r>0$, where $d$ is clear from context. We use standard big-O asymptotic notation: For functions $f,g:\reals\to[0,\infty)$ we write $f=O(g)$ if there exists $c>0$ such that $f(x)\leq c\cdot g(x)$; $f=\Omega(g)$ if $g=O(f)$. We occasionally hide logarithmic factors by writing $f=\widetilde{O}(g)$ if $f=O(g\log(g+1))$.

We call a function $f:\reals^d\to\reals$ $L$-Lipschitz if for any $x,y\in\reals^d:|f(x)-f(y)|\leq L\norm{x-y}$, and $H$-smooth if it is differentiable and $\nabla f:\reals^d\to\reals^d$ is $H$-Lipschitz, namely for any $x,y\in\reals^d:\norm{\nabla f(x)-\nabla f(y)}\leq H\norm{x-y}$.

\paragraph{Nonsmooth analysis.}
By Rademacher's theorem, Lipschitz functions are differentiable almost everywhere (in the sense of Lebesgue). Hence, for any Lipschitz function $f:\reals^d\to\reals$ and point $x\in\reals^d$ the Clarke sub-differential set \citep{clarke1990optimization} can be defined as
\[
\partial f(x):=\conv\{g~:~g=\lim_{n\to\infty}\nabla f(x_n),\,x_n\to x\}
~,
\]
namely, the convex hull of all limit points of $\nabla f(x_n)$ over all sequences of differentiable points which converge to $x$.
Note that if the function is differentiable at a point or convex, the Clarke sub-differential reduces to the gradient or subgradient in the convex analytic sense, respectively.

We say that a point $x$ is an $\epsilon$-stationary point of $f(\cdot)$ if there exists $g\in\partial f(x):\norm{g}\leq\epsilon$.
Furthermore, given $\delta>0$ the Goldstein $\delta$-subdifferential \citep{goldstein1977optimization} of $f$ at $x$ is the set
\[
\partial_{\delta}f(x):=\conv\Big(\bigcup_{z\in B_{\delta}(x)}\partial f(z)\Big)
~,
\]
namely all convex combinations of gradients at points in a $\delta$-neighborhood of $x$. We say that a point $x$ is a $(\delta,\epsilon)$-stationary of $f(\cdot)$ if there exists $g\in\partial_{\delta}f(x)$ such that $\norm{g}\leq\epsilon$. Note that a point is $\epsilon$-stationary if and only if it is $(\delta,\epsilon)$-stationary for all $\delta>0$ \citep[Lemma 7]{zhang2020complexity}.

\paragraph{Algorithms and complexity.}
Throughout this work we consider (possibly randomized) iterative first-order algorithms, from an oracle complexity perspective \citep{nemirovskiyudin1983}. Such an algorithm first produces $x_1$ possibly at random and receives $(f(x_1),\partial f(x_1))$.\footnote{For the purpose of this work it makes no difference whether the algorithm gets to see some subgradient or the whole subdifferential set. That is, the lower bounds to follow hold even if the algorithm has access to the entire subdifferential set, while the upper bounds hold even if the algorithm receives a single subgradient.} Then, for any $t>1$ produces $x_t$ possibly at random based on previously observed responses, and receives $(f(x_t),\partial f(x_t))$. We are interested in the minimal number $T$ for which it can guarantee to produce some $(\delta,\epsilon)$ stationary point. Namely, given some function class $\Fcal$ we look at
\[
\sup_{f\in\Fcal}\min_{T\in\NN}\Pr_{\Acal}[\exists t\in[T]:x_t~\mathrm{is\,}(\delta,\epsilon)\mathrm{-stationary}]\geq\frac{2}{3}
~,
\]
where the probability is with respect to the algorithm's internal randomness when acting upon $f\in\Fcal$ (or deterministically, if the algorithm is deterministic).

\section{Deterministic algorithms} \label{sec: deterministic}

As discussed in the introduction, \citep{zhang2020complexity,davis2021gradient,tian2022finite} all provide randomized algorithms that given any $L$-Lipschitz function $f:\reals^d\to\reals$ and an initial point $x_1$ that satisfies $f(x_1)-\inf_{x}f(x)\leq \Delta$, produce a $(\delta,\epsilon)$-stationary point of $f$ within $\widetilde{O}(\Delta L^2/\delta\epsilon^3)$ oracle calls to $f$. We start by showing that this rate, let alone any dimension-free rate whatsoever, cannot be achieved by any deterministic algorithm.

\begin{theorem} \label{thm: deterministic lower}
For any deterministic first-order algorithm and any iteration budget $T\in\NN$, there exists a $1$-Lipschitz function $f:\reals^d\to\reals,\,d=T+2$ such that $f(x_1)-\inf_{x}f(x)\leq 1$ yet the $T$ iterates produced by the algorithm when applied to $f$ are not $(\delta,\epsilon)$-stationary points for any $\delta\leq\frac{1}{7},\,\epsilon\leq\frac{1}{252}$.
\end{theorem}

We defer the full proof to \secref{sec: proofs}, though the intuition can be described as follows. For any deterministic first-order algorithm, if an oracle can always return the ``uninformative'' answer $f(x_i)=0,\nabla f(x_i)=\e_1$ this fixes the algorithm iterates $x_1,\dots,x_T$. It remains to construct a Lipschitz function that will be consistent with the oracle answers, yet all the queried points are not $(\delta,\epsilon)$-stationary. To that end, we construct a function which in a very small neighborhood of each queried point $x_t$ looks like $x\mapsto \e_1^\top(x-x_t)$, yet in most of the space looks like $x\mapsto\max\{v^\top x,-1\}$ which has $(\delta,\epsilon)$-stationary points only when $x$ is correlated with $-v$. By letting $v$ be some vector which is orthogonal to all the queried points, we obtain the result.

The construction we just described crucially relies on the function being highly nonsmooth - essentially interpolating between two orthogonal linear functions in an arbitrarily small neighborhood.
As it turns out, if the function to be optimized is even slightly smooth, then the theorem above can be bypassed, as manifested in following theorem.

\begin{theorem} \label{thm: deterministic upper}
Suppose $f:\reals^d\to\reals$ is $L$-Lipschitz, $H$-smooth, and $x_1\in\reals^d$ is such that $f(x_1)-\inf_{x}f(x)\leq\Delta$. Then there is a deterministic first-order algorithm that produces a $(\delta,\epsilon)$-stationary point of $f$ within $O(\frac{\Delta L^2\log(H\delta/\epsilon)}{\delta\epsilon^3})$ oracle calls.
\end{theorem}

The simple idea which proves \thmref{thm: deterministic upper} is to replace a certain randomized line search in $\INGD$ with a deterministic binary search subroutine, which terminates within ${O}(\log(H\delta/\epsilon))$ steps if the function is $H$-smooth. Such a procedure is derived by \citet{davis2021gradient} for any $H$-weakly convex function along differentiable directions, and since any $H$-smooth function is $H$-weakly convex and differentiable along any direction, this can be applied as is proving \thmref{thm: deterministic upper}.
Although the algorithmic ingredient to this observation appears inside a proof of \citet{davis2021gradient}, they use it in a different way in order to produce a randomized algorithm for weakly convex functions in low dimension, with different guarantees suitable for that setting. For completeness we explain in \secref{sec: proofs} how this leads to \thmref{thm: deterministic upper}.

\section{Randomized lower bounds} \label{sec: rand lower}

We now turn to establish complexity lower bounds which will hold for any (possibly randomized) algorithm based on a
first-order oracle, as common in the optimization literature stemming from the seminal work of Nemirovski and Yudin \citep{nemirovskiyudin1983}. Our first lower bound will be for the complexity of producing $(\delta,\epsilon)$-stationary points in terms of $\epsilon$-dependence.

\begin{theorem} \label{thm:delta,epsilon eps-lower bound}
For any $\epsilon<1,~\delta\leq\frac{1}{12\epsilon}$ and any randomized first-order algorithm $\Acal$, there exists a $1$-Lipschitz convex function $f:\reals^d\to\reals,~d=\widetilde{O}({1}/{\epsilon^6})$ such that $f(x_1)-\inf_{x}f(x)\leq1$ yet $\Pr_{\Acal}[\exists t\in[T]:x_t~\mathrm{is\,}(\delta,\epsilon)\mathrm{-stationary}]<\frac{1}{3}$ unless $T=\Omega(1/\epsilon^2)$.
\end{theorem}

An observation we would like to make is that \thmref{thm:delta,epsilon eps-lower bound} actually provides an optimal $\Omega(1/\delta\epsilon^3)$ lower bound in a certain parameter regime. Indeed, since it holds even when $\delta=\Omega({1}/{\epsilon})$, for such a choice of $\delta$ we have $\Omega({1}/{\epsilon^2})=\Omega({1}/{\delta\epsilon^3})$, resulting in the optimal bound. In standard settings in optimization which involve optimizing with respect to a single complexity parameter (e.g. $\epsilon$-sub-optimality or $\epsilon$-stationarity) providing a lower bound in a certain parameter regime typically extends to any parameter regime through rescaling
tricks of the form $f(x)\mapsto C_1 f(C_2\cdot x)$. Unfortunately, this does not seem to be the case for $(\delta,\epsilon)$-stationarity as there aren't enough degrees of freedom to control the scales of both $\delta$ and $\epsilon$, while maintaining the Lipschitz and initial sub-optimality constant. 

The proof of \thmref{thm:delta,epsilon eps-lower bound} appears in \secref{sec: proofs}, and is based on the well established machinery of high-dimensional optimization lower bounds (see for example \citet{nemirovskiyudin1983,nesterov2018lectures,woodworth2017lower,carmon2020lowerI}).
To sketch the proof idea, we examine a random orthogonal transformation on top of the so called ``Nemirovski function'', which is approximately of the form $x\mapsto \max\{|x_i-1|\}$.
Though the Nemirovski function is commonly used to prove lower bounds in terms of sub-optimality (i.e. $f(x_t)-\inf_{x}f(x)$), we develop an analysis of its $(\delta,\epsilon)$-stationary points from which a complexity lower bound can be derived.
On the one hand, any oracle response reveals information about a single coordinate which roughly leads to a lower bound of $d$ queries in order to get information about $d$ coordinates. On the other hand, the subgradients at points which are not near $(1,1,\dots,1)\in\reals^d$ are convex combinations of $\pm\e_i$ for some standard basis vectors $\e_i$, thus have norm of at least $\Omega(1/\sqrt{d})$ which is smaller than $\epsilon$ only whenever $d=\Omega(1/\epsilon^2)$.

As to $\delta$-dependent lower bounds, we provide the following theorem.

\begin{theorem} \label{thm: nonconvex 1dim}
For any randomized first-order algorithm $\Acal$ and any $\delta,\epsilon\leq\frac{1}{4}$, there exists a 1-Lipschitz function $f:\reals\to\reals$ such that $f(x_1)-\inf_{x}f(x)\leq 1$, yet $\Pr_{\Acal}[\exists t\in[T]:x_t~\mathrm{is\,}(\delta,\epsilon)\mathrm{-stationary}]<\frac{1}{3}$ unless $T=\Omega(\log(1/\delta))$.
\end{theorem}

Underlying the proof of \thmref{thm: nonconvex 1dim} is a useful observation on the structure of $(\delta,\epsilon)$-stationary points in dimension one: Although in general it is not true that a $(\delta,\epsilon)$-stationary point is necessarily $\delta$-close to an $\epsilon$-stationary point \citep{kornowski2021oracle} (yet the opposite implication is trivially true), for univariate functions these two notions are actually equivalent.

\begin{claim} \label{claim: 1dim_equivalence}
If $f:\reals\to\reals$ is Lipschitz, then $x\in\reals$ is $(\delta,\epsilon)$-stationary if and only if it is $\delta$-close to an $\epsilon$-stationary point.
\end{claim}

We provide a proof of \claimref{claim: 1dim_equivalence} in \secref{sec: proofs}. This claim is useful for our purposes since it allows us to convert any one-dimensional lower bound for getting $\delta$-close to $\epsilon$-stationary points, into a lower bound for finding $(\delta,\epsilon)$-stationary points. The former can be done using a technique which appears in \citep{kornowski2021oracle} which itself is based on a technique for the deterministic setting due to \citet{nemirovski1995lecture}. The basic idea is to construct a function which is parameterized by a binary vector, such that optimizing the function essentially reduces to guessing the binary vector. If the function is constructed such that any oracle response reveals a constant number of bits from the vector, this implies that the length of the vector can serve as a complexity lower bound for the optimization problem.
Since the result in \citep{kornowski2021oracle} is asymptotic - for $\delta\to0$, or equivalently for a binary vector of arbitrary length - providing a specific $\delta$-dependent lower bound boils down to extracting and improving its $\delta$-dependence through a more careful analysis which appears in \secref{sec: proofs}.

\section{$(\delta,\epsilon)$-stationarity for convex functions} \label{sec: convex stat}

Recalling that for convex functions the Clarke sub-differential coincides with the convex subgradient \citep[Proposition 2.2.7]{clarke1990optimization}, we are interested in the possibility of producing points with small subgradients of convex Lipschitz functions.
It turns out that the same discussed proof idea of \thmref{thm: nonconvex 1dim} also extends to convex functions, albeit with substantially more technical work to ensure that global convexity is maintained. This yields the impossibility of producing $\epsilon$-stationary points within any finite time even for convex functions.

\begin{theorem} \label{thm: convex 1dim}
For any randomized first-order algorithm $\Acal$ and any $\delta,\epsilon\leq\frac{1}{4}$, there exists a convex 1-Lipschitz function $f:\reals\to\reals$ such that $f(x_1)-\inf_{x}f(x)\leq 1$, yet $\Pr_{\Acal}[\exists t\in[T]:x_t~\mathrm{is\,}(\delta,\epsilon)\mathrm{-stationary}]<\frac{1}{3}$ unless $T=\Omega(\sqrt{\log(1/\delta)})$. In particular, no finite time algorithm can produce a $\frac{1}{4}$-stationary point of 1-Lipschitz convex functions with constant probability.
\end{theorem}

Following the theorem above, we aim towards achievable relaxations of stationarity for convex Lipschitz functions. One natural relaxation is getting $\delta$-close to an $\epsilon$-stationary point. While it is known that without convexity no algorithm can guarantee getting to such points in a dimension-free rate \citep{kornowski2021oracle}, perhaps not surprisingly, this is not the case for convex functions. Indeed, given a Lipschitz convex function and a diameter bound $\dist(x_1,\arg\min_{x}f(x))\leq R$, \citet{davis2018complexity} provide a first-order algorithm that produces a point which is $\delta$-close to an $\epsilon$-stationary point of $f$ within $T=\widetilde{O}(\frac{R^2}{\delta^2\epsilon^2})$ oracle calls. Furthermore, since any such point is also a $(\delta,\epsilon)$-stationary point, then by \thmref{thm:delta,epsilon eps-lower bound} this is the best achievable rate for this task in terms of $\epsilon$-dependence.

As in the rest of this work, we consider the coarser relaxation of $(\delta,\epsilon)$-stationary points. It is important to emphasize that even for convex functions, 
the two notions of $(\delta,\epsilon)$-stationarity and being  $\delta$-close to an $\epsilon$-stationary point
are inherently different.
Indeed, as we illustrate in \lemref{lemma: delta,epsilon not delta close} in the appendix, a $(\delta,\epsilon)$-stationary point can be arbitrarily far away from any $\epsilon$-stationary point.\footnote{Similar results appear in \citep{kornowski2021oracle,tian2022finite}. However, the former reference provides such a result for nonconvex functions. The latter reference does provides a result for convex functions, but the construction uses a point which is $2\delta$-close to an $\epsilon$-stationary one, rather than arbitrarily far. Accordingly, \lemref{lemma: delta,epsilon not delta close} strengthens both.}
It turns out that for the sake of producing $(\delta,\epsilon)$-stationary points
a simple combination of gradient descent and the $\INGD$ algorithm of \citep{zhang2020complexity} can provide a better dependence in terms of $\delta$ than \citep{davis2018complexity} as we claim next.


\begin{theorem} \label{thm:delta,epsilon upper bound}
Suppose $f:\reals^d\to\reals$ is $L$-Lipschitz and convex, and $x_1\in\reals^d$ is such that $\dist(x_1,\arg\min_x f(x))\leq R$. Then there exists a randomized first-order algorithm that produces a $(\delta,\epsilon)$-stationary point of $f$ with probability at least $\frac{2}{3}$ within $T=\widetilde{O}\left(\frac{L^2 R^{2/3}}{\delta^{2/3}\epsilon^{2}}\right)$ oracle calls. Moreover, if $f$ is also $H$-smooth, then there exists such a deterministic algorithm with $T=\widetilde{O}\left(\frac{L^2 R^{2/3}\log(H)}{\delta^{2/3}\epsilon^{2}}\right)$.
\end{theorem}

The proof is provided in \secref{sec: proofs}. We would like to point out is that the improved rates for convex functions provided by \thmref{thm:delta,epsilon upper bound} require the stronger assumption of a domain bound $R:=\dist(x_1,\arg\min_{x}f(x))$, as opposed to a sub-optimality bound $\Delta:=f(x_1)-\inf_{x}f(x)$. This is indeed a stronger assumption since for any $L$-Lipschitz function we have $\Delta\leq LR$. Furthermore, since the lower bound in \thmref{thm:delta,epsilon eps-lower bound} holds for convex functions, and as we explained why it provides an $\Omega(1/\delta\epsilon^3)$ lower bound for a certain parameter regime,
we cannot expect such an improvement in general without any further assumption. We note that this is analogous to the smooth convex case for which a gap between $\Delta$-based rate and $R$-based rate exists as well \citep{carmon2021lowerII}.

\section{Discussion} \label{sec: discussion}

In this paper, we studied the oracle complexity of producing $(\delta,\epsilon)$-stationary points in various settings. Following our results, there are several notes and open ends we would like to point out.

In our opinion, the main open question in this realm is narrowing the gap between the $\widetilde{O}(1/\delta\epsilon^3)$ complexity upper bound of the $\INGD$ algorithm, and our $\Omega(1/\epsilon^2+\log(1/\delta))$ complexity lower bound (which follows by combining \thmref{thm:delta,epsilon eps-lower bound} and \thmref{thm: nonconvex 1dim}). We observed that the $\Omega(1/\epsilon^2)$ lower bound in \thmref{thm:delta,epsilon eps-lower bound} actually matches the upper bound in a certain parameter regime, which we suspect hints that the lower bound can be further improved. This would require in particular a polynomial lower bound in terms of $\delta^{-1}$ which we do not currently know how to prove.

We also examined the complexity of producing small subgradients of convex Lipschitz functions, showing that no finite algorithm can obtain such points unless a positive $\delta$-relaxation is introduced, similarly to the nonconvex case. Given an additional assumption of a domain bound as opposed to a sub-optimality bound, we were able to provide improved rates for convex functions.
We would like to point out that although the lower bound in \thmref{thm:delta,epsilon eps-lower bound} is stated in terms of a sub-optimality bound, the same proof can be easily adapted to the bounded domain case. Indeed, by replacing $|x_i-1|$ in the proof by $|x_i-R/\sqrt{d}|$, the same result holds for any $\delta=O(R)$.

Another interesting direction for future work is to establish dimension-dependent lower bounds for producing $(\delta,\epsilon)$-stationary points. Although several recent works considered this setting \citep{davis2021gradient,lin2022gradientfree}, no complementing lower bounds currently exist.

\section{Proofs} \label{sec: proofs}

\subsection{Proof of \thmref{thm: deterministic lower}}

Fix $T\in\NN$.
Suppose that for any $i\in[T-1]$ the first-order oracle response is $f(x_i)=0,\nabla f(x_i)=\e_1$. Since the algorithm is deterministic this fixes the iterate sequence $x_1,\dots,x_T$.
We will show this resisting strategy is indeed consistent with a function which satisfies the conditions in the theorem.

To that end, we denote $r:=\min_{1\leq i\neq j\leq T}\norm{x_i-x_j}/4$ and fix some $v\in(\mathrm{span}\{\e_1,x_1,\dots,x_T\})^\perp$ with $\norm{v}=1$ (which exists since $d=T+2$). For any $z\in\reals^d$ we define
\[
g_{z}(x):=\min\{\norm{x-z}^2/r^2,1\}v^\top x+(1-\min\{\norm{x-z}^2/r^2,1\})\e_1^\top(x-z)~,
\]
and further define
\[
h(x):=\begin{cases}
v^\top x\,,&\forall i\in[T]:\norm{x-x_i}\geq r
\\
g_{x_i}(x)\,,&\exists i\in[T]:\norm{x-x_i}< r
\end{cases}
~.
\]
Note that $h$ is well defined since by definition of $r$ there cannot be $i\neq j$ such that $\norm{x-x_i}<r$ and $\norm{x-x_j}<r$.

\begin{lemma}
$h:\reals^d\to\reals$ as defined above is $7$-Lipschitz, satisfies for any $i\in[T]:\,h(x_i)=0,\nabla h(x_i)=\e_1$ and has no $(\delta,\frac{1}{36})$-stationary points for any $\delta>0$.
\end{lemma}

\begin{proof}
We start by noting that $h$ is continuous, since for any $z$ and $(y_n)_{n=1}^{\infty}\subset B_r(z),\,y_n\overset{n\to\infty}{\longrightarrow}y$ such that $\norm{y-z}=r$ we have
\begin{align*}
\lim_{n\to\infty}h(y_n)
&=\lim_{n\to\infty}g_z(y_n)
\\&=\lim_{n\to\infty}\left(\min\{\norm{y_n-z}^2/r^2,1\}v^\top y_n+(1-\min\{\norm{y_n-z}^2/r^2,1\})\e_1^\top(y_n-z)\right)
\\&=\lim_{n\to\infty}\left(\frac{\norm{y_n-z}^2}{r^2}\cdot v^\top y_n+\left(1-\frac{\norm{y_n-z}^2}{r^2}\right)\e_1^\top(y_n-z)\right)
\\&=v^\top y
~.
\end{align*}
Having established continuity, since $x\mapsto v^{\top}x$ is clearly $1$-Lipschitz (in particular $7$-Lipschitz), in order to prove Lipschitzness of $h$ it is enough to show that $g_{x_i}(x)$ is $7$-Lipschitz in $\norm{x-x_i}<r$ for any $x_i$. For any such $x,x_i$ we have
\begin{align}
\nabla g_{x_i}(x)
&~=~\frac{2v^\top x}{r^2}(x-x_i)+\frac{\norm{x-x_i}^2}{r^2}v
-\frac{2\e_1^\top(x-x_i)}{r^2}(x-x_i)
-\frac{\norm{x-x_i}^2}{r^2}\e_1
+\e_1 \nonumber
\\
&\overset{v\perp x_i}{=}~
\frac{2v^\top (x-x_i)}{r^2}(x-x_i)+\frac{\norm{x-x_i}^2}{r^2}v
-\frac{2\e_1^\top(x-x_i)}{r^2}(x-x_i)
-\frac{\norm{x-x_i}^2}{r^2}\e_1
+\e_1
~,
\label{eq: nabla g}
\end{align}
hence
\begin{align*}
\norm{\nabla g_{x_i}(x)}
&=\norm{\frac{2v^\top (x-x_i)}{r^2}(x-x_i)+\frac{\norm{x-x_i}^2}{r^2}v-\frac{2\e_1^\top(x-x_i)}{r^2}(x-x_i)
-\frac{\norm{x-x_i}^2}{r^2}\e_1+\e_1}
\\
&\leq\frac{2\norm{v}\cdot \norm{x-x_i}^2}{r^2}+\frac{\norm{x-x_i}^2}{r^2}\norm{v}
+\frac{2\norm{\e_1}\cdot\norm{x-x_i}^2}{r^2}
+\frac{\norm{x-x_i}^2}{r^2}\norm{\e_1}
+\norm{\e_1}
\\
&\leq2+1+2+1+1=7~,
\end{align*}
which proves the desired Lipschitz bound.
The fact that for any $i\in[T]:h(x_i)=0,\,\nabla h(x_i)=\mathbf{e_1}$ is easily verified by construction and by \eqref{eq: nabla g}.
In order to finish the proof, we need to show that $h$ has no $(\delta,\frac{1}{36})$ stationary-points. 
By construction we have
\[
\partial h(x)=
\begin{cases}
v\,,&\forall i\in[T]:\norm{x-x_i}>r
\\
\nabla g_{x_i}(x)\,,&\exists i\in[T]:\norm{x-x_i}< r
\end{cases}
~,
\]
while for $\norm{x-x_i}=r$ we would get convex combinations of the two cases.\footnote{Since we are interested in analyzing the $\delta$-subdifferential set which consists of convex combinations of subgradients, and subgradients are defined as convex combinations of gradients at differentiable points - we do not lose anything by considering convex combinations gradients at differentiable points in the first place.}
Inspecting the set $\{\nabla g_{x_i}(x):\norm{x-x_i}<r\}$ through \eqref{eq: nabla g}, we see that it depends on $x,x_i$ only through $x-x_i$ and that actually
\[\{\nabla g_{x_i}(x):\norm{x-x_i}<r\}=\{\nabla g_{\vec{0}}(x):\norm{x}<r\}~,
\]
which is convenient since the latter set does not depend on $x_i$. Overall, we see that any gradient of $h$ is in the set
\begin{align*}
\mathcal{G}:=&\left\{
\lambda_1 v+ \lambda _2\left(\frac{2v^\top x}{r^2}x+\frac{\norm{x}^2}{r^2}v
-\frac{2\e_1^\top x}{r^2}x
-\frac{\norm{x}^2}{r^2}\e_1
+\e_1\right)
~:~\lambda_1,\lambda_2\geq0,\lambda_1+\lambda_2=1,\norm{x}\leq r
\right\}
\\
=&\left\{
\lambda_1 v+ \lambda _2\left({2v^\top x}\cdot x+{\norm{x}^2}\cdot v
-{2\e_1^\top x}\cdot x
-{\norm{x}^2}\cdot\e_1
+\e_1\right)
~:~\lambda_1,\lambda_2\geq0,\lambda_1+\lambda_2=1,\norm{x}\leq 1
\right\}
\\
=&\left\{
(\lambda_1+\lambda_2\norm{x}^2)\cdot v
+2\lambda_2((v-\e_1)^\top x)\cdot x
+\lambda_2(1-\norm{x}^2)\cdot \e_1
~:~\lambda_1,\lambda_2\geq0,\lambda_1+\lambda_2=1,\norm{x}\leq 1
\right\}~.
\end{align*}
We aim to show that $\mathrm{conv}(\mathcal{G})$ does not contain any vectors of norm smaller than $\frac{1}{36}$. For $u\in\mathcal{G}$ with corresponding $\lambda_1,\lambda_2,x$ as above, it holds that
\begin{align} \label{eq: uv}
u^\top v
&=\lambda_{1}+\lambda_{2}\norm{x}^2+2\lambda_2(v-e_1)^\top x\cdot x^\top v 
\nonumber\\
&=\lambda_{1}+\lambda_{2}\norm{x}^2+2\lambda_2 (v^\top x)^2-2\lambda_2e_1^\top x\cdot x^\top v
\nonumber\\
&\geq \lambda_{1}+\lambda_{2}(v^\top x)^2+\lambda_{2}(e_1^\top x)^2
+2\lambda_2 (v^\top x)^2-2\lambda_2e_1^\top x\cdot x^\top v
\nonumber\\
&=\lambda_{1}+\lambda_2(v^\top x- e_1^\top x)^2
+2\lambda_2 (v^\top x)^2
\nonumber\\
&\geq \lambda_{1}+\lambda_2(v^\top x- e_1^\top x)^2
\nonumber\\
&\geq \lambda_2(v^\top x- e_1^\top x)^2
~.
\end{align}
So for $\xi\in\conv(\mathcal{G})$ represented as the convex combination $\xi=\sum_{i=1}^{N}\mu^i u^i$, with each $u^i\in\mathcal{G}$ having its corresponding $\lambda_1^i,\lambda_2^i,x^i$, we get

\begin{align} \label{eq: uv summed}
\xi^\top v&=\sum_{i=1}^{N}\mu^i (u^i)^\top v
\overset{(\ref{eq: uv})}{\geq}
\sum_{i=1}^{N}\mu^i \lambda_2^i((v-e_1)^\top x^i)^2
\overset{\text{Cauchy-Schwarz}}{\geq} \frac{\left(\sum_{i=1}^{N}\mu^i \lambda_2^i|(v-e_1)^\top x^i|\right)^2}{\sum_{i=1}^{N}\mu^i\lambda_2^i}
\nonumber
\\
&\geq
\left(\sum_{i=1}^{N}\mu^i \lambda_2^i|(v-e_1)^\top x^i|\right)^2
~,
\end{align}
where in the last inequality we used $\sum_{i=1}^N\mu^i\lambda_2^i\leq \max_{i}\lambda_2^i\cdot\sum_i\mu^i\leq1$. 
This further gives
\begin{align*}
\xi^\top(e_1+v)
&=\sum_{i=1}^{N}\mu^i (e_1+v)^\top u^i
\\&=\sum_{i=1}^{N}\mu^i [1+2\lambda^i_2(v^\top x^i-e_1^\top x^i)(e_1^\top x^i+v^\top x^i)]
\\
&=1+2\sum_{i=1}^{N}\mu^i \lambda^i_2((v-e_1)^\top x^i)((v+e_1)^\top x^i)
\\
&\overset{\|x^i\|\leq 1}{\geq}
1-4\sum_{i=1}^{N}\mu^i \lambda^i_2|(v-e_1)^\top x^i|
\\
&\overset{(\ref{eq: uv summed})}{\geq} 1-4\sqrt{\xi^\top v}
~.
\end{align*}
Hence,
given $\xi\in\conv(\mathcal{G})$ we can assume that $\|\xi\|\leq 1$ (since otherwise there is nothing left to show), and see that
\begin{gather*}
1
\leq |\xi^\top (e_1+v)|+4\sqrt{\lambda_2 \xi^\top v}
\leq \norm{\xi}\cdot \norm{e_1+v}+4\sqrt{\norm{\xi}}
\leq \sqrt{2}\norm{\xi}+4\sqrt{\norm{\xi}}
\overset{\norm{\xi}<1}{\leq} \sqrt{2\norm{\xi}}+4\sqrt{\norm{\xi}}
\\
\implies
\norm{\xi}\geq \frac{1}{(\sqrt{2}+4)^2}
>\frac{1}{36}~.
\end{gather*}

\end{proof}

Given the previous lemma we can easily finish the proof of the theorem by looking at
\[
f(x):=\frac{1}{7}\max\{h(x),-1\}~.
\]
$f$ is $1$-Lipschitz (since $h$ is $7$-Lipschitz), and satisfies $f(x_1)-\inf_{x}f(x)=0-(-\frac{1}{7})<1$. Furthermore, for any $i\in[T]: h(x_i)=0>-1\implies f(x_i)=\frac{1}{7}h(x_i)=0$, which by the fact that $h$ is $7$-Lipschitz implies that for any $x\in B_{\frac{1}{7}}(x_i): h(x)>-1\implies \partial f(x)=\frac{1}{7}\partial h(x)$. In particular, 
$\partial_{\frac{1}{7}} f(x_i)=\frac{1}{7}\partial_{\frac{1}{7}} h(x_i)$
so we conclude using the lemma shows that no $x_i$ is a $(\frac{1}{7},\frac{1}{7}\cdot\frac{1}{36})=(\frac{1}{7},\frac{1}{252})$ stationary point of $f$.

\subsection{Proof of \thmref{thm: deterministic upper}}

The proof is based on a slight modification of the $\INGD$ algorithm of \citet{zhang2020complexity}. To give some background, this algorithm is based on a gradient-descent-like procedure, where in each iteration, given current point $x_t$, it uses a certain element $g_t\in \partial_{\delta}f(x_t)$, until it finds a $g_t$ such that $\norm{g_t}\leq \epsilon$. If $g_t$ does not satisfy this condition, it is ensured that the function value decreases by at least $\delta\epsilon/4$, which can happen at most $4\Delta/\delta\epsilon$ times. $g_t$ itself is produced by repeatedly searching in various directions away from $x_t$, and maintaining a convex combination of elements in $\partial_{\delta}f(x)$, until either a sufficient decrease is found or the norm of the convex combination becomes sufficiently small. 

In the original $\INGD$ algorithm, a given direction is searched by picking a point at random in an interval starting from $x_t$. If there is no sufficient decrease in that direction, the fundamental theorem of calculus implies that the derivative along that direction is small on average, hence picking a point at random is sufficient. In our case, we de-randomize this part of the algorithm, by replacing it with a deterministic binary search (following \citet{davis2021gradient}). 

More formally, we replace the $\INGD$ algorithm with Algorithms \ref{alg: INGD},\ref{alg: min norm},\ref{alg: binary search} below, where the change is in calling the binary search subroutine (Algorithm \ref{alg: binary search}) instead of picking a point at random. Inspecting Algorithm \ref{alg: min norm}, we see that it maintains a convex combination of elements in $\partial_{\delta}f(x)$, and calls the subroutine $\mathrm{BinarySearch}(x,g_k)$ only whenever $\norm{g_k}>\epsilon$ and $f(x)-f(x-\frac{\delta}{\norm{g_k}}g_k)\leq\frac{\delta}{4}\norm{g_k}$. That being the case, the fundamental theorem of calculus ensures
\begin{align*}
\frac{1}{4}\norm{g_k}^2
\geq \frac{\norm{g_k}}{\delta}\left({f(x)-f\left(x-\frac{\delta}{\norm{g_k}}g_k\right)}\right)
=\frac{1}{\delta}\int_{0}^{\delta}\inner{\nabla f\left(x-\frac{t}{\norm{g_k}}g_k\right),g_k}dt
~.
\end{align*}
Namely, on average along the segment $[x,x-\frac{\delta}{\norm{g_k}}g_k]$, points have gradients whose dot product with $g_k$ is small.
Accordingly, using the fact the $\nabla f$ is $H$-Lipschitz,
$\mathrm{BinarySearch}(x,g_k)$ performs a binary search along this segment and
produces a point $y_k$ such that $\langle\nabla f(y_k),g_k\rangle\leq \frac{1}{2}\norm{g_k}^2$
within $O(\log(\delta H/\norm{g_k}))=O(\log(\delta H/\epsilon))$ first-order oracle calls \citep[Lemma 3.9]{davis2021gradient}.
This is the only modification to $\INGD$ which is introduced to Algorithm \ref{alg: INGD}, when compared to \citep{zhang2020complexity,davis2021gradient}.
Accordingly, $\mathrm{MinNorm}$ finds an approximate minimal norm element of $\partial_{\delta}f(x_t)$ within $O\left(\frac{L^2\log(\delta H/\epsilon)}{\epsilon^2}\right)$ first-order oracle calls \citep[Corollary 2.5]{davis2021gradient}. Since we claimed that this can happen at most $O(\Delta/\delta\epsilon)$ times,
overall this produces a $(\delta,\epsilon)$-stationary point of $f$ within $O\left(\frac{\Delta L^2\log(\delta H/\epsilon)}{\delta\epsilon^3}\right)$ oracle calls.

\begin{algorithm}
\caption{$\INGD(x_1,T)$}\label{alg: INGD}
\begin{algorithmic}
\Require Initialization $x_1\in\reals^d$, iteration budget $T\in\NN$.
\For{$t=1,\dots,T-1$}
\State $g_t\gets \mathrm{MinNorm}(x_t)$ \Comment{Complexity $O(\frac{L^2\log(\delta H/\epsilon)}{\epsilon^2})$}
\If{$\norm{g_t}\leq \epsilon$}
\State \Return $x_t$
\Else
\State $x_{t+1}\gets x_t-\frac{\delta}{\norm{g_t}}g_t$
\EndIf
\EndFor
\State \Return $x_T$
\end{algorithmic}
\end{algorithm}

\begin{algorithm}
\caption{$\mathrm{MinNorm}(x)$}\label{alg: min norm}
\begin{algorithmic}
\Require $x\in\reals^d,~\delta,\epsilon>0$.
\State $k \gets 0$
\State $g_0 \gets \nabla f(x)$
\While{$\norm{g_k}>\epsilon$ and $f(x)-f(x-\frac{\delta}{\norm{g_k}}g_k)\leq\frac{\delta}{4}\norm{g_k}$}
\State $y_k\gets \mathrm{BinarySearch}(x,g_k)$ \Comment{Finds $y_k\in[x,x-\frac{\delta}{\norm{g_k}}g_k]:
\langle\nabla f(y_k),g_k\rangle\leq \frac{1}{2}\norm{g_k}^2$}
\State $g_{k+1}\gets \arg\min_{\lambda\in[0,1]}\norm{\lambda g_k+(1-\lambda)\nabla f(y_k)}$
\State $k\gets k+1$
\EndWhile
\State \Return $g_k$
\end{algorithmic}
\end{algorithm}

\begin{algorithm}
\caption{$\mathrm{BinarySearch}(x,g)$}\label{alg: binary search}
\begin{algorithmic}
\Require  $x,g\in\reals^d,~\delta,H>0$.
\State $(a,b)\gets (0,1)$
\State $\bar{g}\gets \frac{\delta}{\norm{g}}g$
\While{$b-a>\frac{\norm{g}}{8\delta H}$}
\If{$f(x-\frac{a+b}{2}\cdot\bar{g})\geq\frac{1}{2}(f(x-a\bar{g})+f(x-b\bar{g}))$}
\Comment{$\E_{t\in[a,\frac{a+b}{2}]}\langle \nabla f(z-t\bar{g}),g\rangle\leq \E_{t\in[\frac{a+b}{2},b]}\langle \nabla f(z-t\bar{g}),g\rangle$}
\State $(a,b)\gets (a,\frac{a+b}{2})$
\Else
\Comment{$\E_{t\in[a,\frac{a+b}{2}]}\langle \nabla f(z-t\bar{g}),g\rangle> \E_{t\in[\frac{a+b}{2},b]}\langle \nabla f(z-t\bar{g}),g\rangle$}
    \State $(a,b)\gets (\frac{a+b}{2},b)$
\EndIf
\EndWhile
\State \Return $x-a\bar{g}$
\end{algorithmic}
\end{algorithm}

\subsection{Proof of \thmref{thm:delta,epsilon eps-lower bound}}

We derive the lower bound using the so-called ``zero chain'' technique \citep{carmon2020lowerI,woodworth2017lower}. While a full discussion of this technique can be found in \citep{carmon2020lowerI,carmon2020complexity}, we will only state the relevant definitions and propositions we will use in our proof. 

\begin{definition}
For $x\in\reals^d,$ we define its $\alpha$-progress as $\prog_\alpha(x):=\max\{i\in[d]:|x_i|>\alpha\}$.
\end{definition}

\begin{definition}
A function $f:\reals^d\to\reals$ is an $\alpha$-robust zero-chain if for every $x\in\reals^d$:
\[
\prog_{\alpha}(x)<i
\implies f(y)=f(y_1,\dots,y_i,0,\dots,0)
~~\textit{for~all~y~in~a~neighborhood~of~x.}
\]
\end{definition}

\begin{proposition}[\cite{carmon2020complexity}, Proposition 2.3] \label{prop: Carmon's prog}
Let $R,\alpha>0,\,T\in\NN$ and $d'\geq\lceil T+\frac{2R^2}{\alpha^2}\log(3T^2)\rceil$. Let $f:\reals^T\to\reals$ be an $\alpha$-robust zero-chain, let $U\in\reals^{d'\times T}$ be a uniformly random orthogonal matrix and denote $f_U(x)=f(U^T x)$. If $\Acal$ is a (possibly randomized) algorithm interacting with $f_U$ such that its iterates satisfy $\norm{x^{(t)}}\leq R$ for all $t$ with probability 1, then with probability at least $\frac{2}{3}$ over the draw of $U$:
\[
\prog_{\alpha}(U^T x^{(t)})< t ~~\textit{for~all~}t\leq T.
\]
\end{proposition}

Given the proposition, we aim to construct a robust zero chain function which such that all its $(\delta,\epsilon)$-stationary points will have large progress. To that end, for any $T\in\NN,\alpha>0$ we define $f:\reals^{T}\to\reals$ as follows
\begin{equation} \label{eq: Nemirovski func}
f_{T,\alpha}(x):=
\max_{i\in [T]}\{|x_i-1|+3\alpha(T-i)\}    
\end{equation}

\begin{lemma}
For any $T\in\NN,\alpha>0$, $f_{T,\alpha}$ is $1$-Lipschitz, convex and is an $\alpha$-robust zero-chain.
\end{lemma}
The lemma and its proof are similar to \citep[Lemma 23]{carmon2020acceleration}, though we provide a proof for completeness.

\begin{proof}
Let $T\in\NN,\alpha>0$ and abbreviate $f=f_{T,\alpha}$. The Lipschitz property is easily seen since is a composition of $1$-Lipschitz functions. Convexity follows from the fact that the maximum of convex functions is itself convex. To prove that $f$ is an $\alpha$-robust zero-chain, let $x\in\reals^T$ be such that $\prog_{\alpha}(x)<i$, and $y\in B_{\alpha}(x)$. Note that $|x_i|\leq \alpha$ and that
$\forall j:|x_j-y_j|<\alpha$. So for any $i+1\leq j\leq T$ we get 
\begin{align*}
|y_j-1|+3\alpha(T-j)
&\leq |x_j-1|+3\alpha(T-j)+\alpha
\leq 1+2\alpha+3\alpha(T-j)
\leq 1+2\alpha+3\alpha(T-i-1)
\\
&=1-\alpha+3\alpha(T-i)
\leq |x_i-1|-\alpha+3\alpha(T-i)
\leq |y_i-1|+3\alpha(T-i)
~.
\end{align*}
By \eqref{eq: Nemirovski func} this shows $f$ is an $\alpha$-robust zero-chain by definition.
\end{proof}

From now on we fix some $\epsilon<1,\,\delta\leq\frac{1}{12\epsilon}$
and set $T=\frac{1}{4\epsilon^2},~\alpha=\frac{1}{9T}$. The following lemma provides a crucial property of the $(\delta,\epsilon)$-stationary point of $f_{T,\alpha}$.

\begin{lemma} \label{lemma: delta,epsilon progress}
Let $x$ be a $(\delta,\epsilon)$-stationary point of $f_{T,\alpha}$. Then there exists $I\subset[T],|I|\geq\frac{3T}{4}$ such that $\forall i\in I: x_i\geq\frac{1}{3}$. In particular $\prog_{\alpha}(x)\geq\frac{3T}{4}$.

\end{lemma}

\begin{proof}
Throughout the proof we abbreviate $f=f_{T,\alpha}$. Denote $f_i(x):=|x_i-1|+3\alpha(T-i)$, and note that $f(x)=\max_{i\in[T]}\{f_i(x)\}$. Consequently,
\begin{align*}
\partial f(x)&=\conv\{\partial f_i(x)~:~f_i(x)=f(x)\}
\\
&=\conv\{ g_i(x)~:~f_i(x)=f(x)\}~,
~~g_{i}(x):=\begin{cases}
\{\sign(x_i-1)\mathrm{e}_{i}\}, & x_i\neq 1\\
\{s\cdot\mathrm{e}_{i}:-1\leq s\leq1\}, & x_i=1
\end{cases}
\\
\implies
\partial_{\delta}f(x)&=
\conv\{ g_i(y)~:~f_i(y)=f(y),\,\norm{x-y}<\delta\}~.
\end{align*}
From the representation above we can see that if $x$ satisfies that
\begin{equation} \label{eq:box condition}
    \forall y\in B_{\delta}(x)~\forall i\in[T]
    \textit{~such~that~}f_i(y)=f(y)~
    :~|y_i-1|>0
\end{equation}
then any $g\in\partial_{\delta}f(x)$ is a certain convex combination of vectors of the form $\pm \mathrm{e}_i$ (since in that case the second condition in the definition of $g_i$ is never satisfied).
Furthermore, in this case note that either $\mathrm{e}_i\in\partial_{\delta}f(x)$ or $-\mathrm{e}_i\in\partial_{\delta}f(x)$ but not both, since if both of them are in $\partial_{\delta}f(x)$ then by convexity of $\partial_{\delta}$ there exists some $y_i-1=0$ contradicting \eqref{eq:box condition}.
We will now show that if that is the case, then $x$ is \emph{not} a $(\delta,\epsilon)$-stationary point. Indeed, if $x$ satisfies \eqref{eq:box condition} then for any $g\in\partial f_{\delta}(x)$ there exist $S=\{i_1,\dots,i_{|S|}\}\subseteq[T],\lambda_{i_1},\dots,\lambda_{i_{|S|}}\geq0:\sum_{i\in S}\lambda_i=1$ and $\sigma_{i_1},\dots,\sigma_{i_{|S|}}\in\{\pm1\}$ such that $g=\sum_{i\in S}\lambda_{i}\cdot \sigma_i \mathrm{e}_i$. Hence 
\[
\norm{g}
=\sqrt{\sum_{i\in S}\lambda_i^2}
\overset{\mathrm{Cauchy\,Schwarz}}{\geq}
\left(\sum_{i\in S}\frac{1}{\sqrt{|S|}}\lambda_i\right)
=\frac{1}{\sqrt{|S|}}\geq\frac{1}{\sqrt{T}}=2\epsilon~.
\]
This establishes that if $x$ satisfies \eqref{eq:box condition} then it is indeed not a $(\delta,\epsilon)$-stationary point. So if $x$ is a $(\delta,\epsilon)$-stationary point, we deduce that there exist $y\in B_{\delta}(x),i\in[T]$ such that $f_i(y)=f(y)$ and also $y_i=1$. For this $y$ it holds that $\forall j\in[T]:f_j(y)\leq f(y)=f_i(y)$, hence
\begin{gather*}
|y_j-1|+3\alpha(T-j)
\leq |y_i-1|+3\alpha(T-i)
=3\alpha(T-i)
\\
\implies |y_j-1|\leq 3\alpha(j-i)\leq 3\alpha T\leq\frac{1}{3}~.
\end{gather*}
We see that for all $j\in[T]:\,\frac{2}{3}\leq y_j\leq\frac{4}{3}$,
and recall that $\norm{x-y}<\delta$. Assuming towards contradiction that the lemma is not true, we get $|J|\subset[T],\,|J|\geq\frac{T}{4}$ such that $\forall j\in J:x_j<\frac{1}{3}$ and in particular $|x_j-y_j|>\frac{1}{3}$. But this means that $\norm{x-y}>\frac{1}{3}\cdot\sqrt{\frac{T}{4}}=\frac{\sqrt{T}}{6}=\frac{L}{12\epsilon}\geq\delta$ which is a contradiction.
\end{proof}

Given the previous lemma, finishing up the proof is done by a standard application of the aforementioned lower bound techniques. All we need to do is to apply a random orthogonal transformation on top of our constructed function, and to
handle ``large queries'' by extending our function such that points of large norm do not depend on the orthogonal transformation (see \citep[Section 5]{carmon2020lowerI}). 
We do this by defining
\[
\widetilde{f}_U(x):=
\max\{f_{T,\alpha}(U^{T}x),~2(\norm{x}-2\sqrt{T})\}
~.
\]
On the one hand, if $\norm{x}< 2\sqrt{T}$ then
\[
2(\norm{x}-2\sqrt{T})< 0 \leq f_{T,\alpha}(U^{T}x)
\implies \widetilde{f}_U(x)=f_{T,\alpha}(U^{T}x)
~.
\]
On the other hand, if $\norm{x}>6\sqrt{T}$ then
\begin{gather*}
f_{T,\alpha}(U^{T}x)
\leq 
\norm{Ux}+\|\vec{1}\|+\frac{1}{3}
<
(\norm{x}+2\sqrt{T})
+(\norm{x}-6\sqrt{T})
=2(\norm{x}-2\sqrt{T})
\\
\implies \widetilde{f}_U(x)=2(\norm{x}-2\sqrt{T})
~,
\end{gather*}
which does not depend on $U$. Furthermore, repeating essentially the same calculation as in \lemref{lemma: delta,epsilon progress} shows that the $(\delta,\epsilon)$-stationary points of $\widetilde{f}_U(\cdot)$ are exactly the $(\delta,\epsilon)$-stationary points of $f(U^T\cdot)$, which in particular shows they satisfy $\prog_{\alpha}(U^T x)\geq \frac{3T}{4}$. Applying \propref{prop: Carmon's prog} with $R=6\sqrt{T}$ finishes the proof.

\subsection{Proof of \thmref{thm: convex 1dim}}

Let $\delta\leq\frac{1}{4},~T<\frac{1}{6}\sqrt{\log(1/\delta)}$.
By Yao's lemma \citep{yao1977probabilistic}, we may assume $\Acal$ is deterministic and provide a distribution over hard functions.
Namely, we will construct an index set $\Sigma$ and a parameterized subset of convex Lipschitz functions $\{f_{\sigma}:\sigma\in\Sigma\}$ that satisfy for any $\sigma\in\Sigma:$
\begin{itemize}
    \item $f_\sigma$ is convex and $1$-Lipschitz.
    \item $f_\sigma$ attains its minimum with $\arg\min_{x} f(x)\subset(0,1)$. In particular, by Lipschitzness $f(x_1)-\inf_{x}f(x)\leq1$.
    \item $\forall x\notin\arg\min_{x}f(x),~\forall f'(x)\in\partial f(x):~|f'(x)|\geq\frac{1}{4}$.
\end{itemize}
Yet, the first $T$ iterates produced by $\Acal(f_\sigma):x_{1}^{f_\sigma},\dots,x_{T}^{f_\sigma}$, satisfy 
\begin{equation} \label{eq: Yao hard distribution}
\Pr_{\sigma\sim\Dcal}\left[\exists t\in[T]:
\dist(x_{t}^{f_\sigma},\arg\min_x f_\sigma(x))<\delta\right]<\frac{2}{3}~,
\end{equation}
for some distribution $\Dcal$ over $\Sigma$.
In particular, by the third bullet above, with the same probability and \claimref{claim: 1dim_equivalence} they are not $(\delta,\frac{1}{4})$-stationary.

We start by defining a sequence of segments in $[0,1]$ parameterized by binary vectors,
which will form the basis for our upcoming function class construction.
For any $i\in\NN$ we denote
\begin{align*}
I^i_0 &:=\left(\frac{1}{2}-\frac{2}{8^{i+1}},\frac{1}{2}-\frac{1}{8^{i+1}}\right)\subset (0,1)
\\
I^i_1 &:= \left(\frac{1}{2}+\frac{1}{8^{i+1}},\frac{1}{2}+\frac{2}{8^{i+1}}\right)\subset (0,1)
~.
\end{align*}
We interpret the subscript $0$ as "left" and 1 as "right", corresponding to the location of the segment with respect to the center $\frac{1}{2}$. 
Denote by $\phi_{0}^i,\phi_{1}^i$ the unique affine functions with positive derivatives which map $(0,1)$ to $I^i_{0},I^i_{1}$ respectively - which we think of as ``rescaling'' the unit segment to the left or to the right.
We can now define for any $k\in\NN,\,(\sigma_1,\dots,\sigma_k)\in\{0,1\}^k:$

\[
I_{\sigma_1,\dots,\sigma_k}:=
\phi^{1}_{\sigma_{1}}\circ\cdots\circ\phi^{k}_{\sigma_{k}}(0,1)
\subset (0,1)
~.
\]

The following lemma shows that the parameterization of the segments above forms a binary tree structure. This correspondence is given by considering binary vectors as vertices, while its sub-tree consists of vectors for which it serves as a prefix. Later on we will use this intuition to construct functions parameterized by binary vectors, in a way that essentially reduces optimization to traversing the tree. In particular, optimizing the function will require getting to a randomly selected leave which is unknown in advance to the algorithm.

\begin{lemma} \label{lemma: I properties}
\begin{enumerate}
    \item For any $l<k$ and any $\sigma_{1},\dots,\sigma_{l},\dots,\sigma_{k}\in\{0,1\}: I_{\sigma_1,\dots,\sigma_l}\supset I_{\sigma_1,\dots,\sigma_{l},\dots,\sigma_k}$.

    \item If $(\sigma_1,\dots,\sigma_{k-1})\neq(\sigma'_1,\dots,\sigma'_{k-1})$ then for all $\sigma_k,\sigma'_k\in\{0,1\}:\ I_{\sigma_1,\dots,\sigma_k}\cap I_{\sigma'_1,\dots,\sigma'_k}=\emptyset$. 
\end{enumerate}

\end{lemma}

\begin{proof}
\begin{enumerate}
    \item $I_{\sigma_1,\dots,\sigma_l,\dots,\sigma_k}
=\phi^{1}_{\sigma_{1}}\circ\cdots\circ \phi^{l}_{\sigma_{l}}(
\phi^{l+1}_{\sigma_{l+1}}\circ\cdots
\circ\phi^{k}_{\sigma_{k}}(0,1))
\subset\phi^{1}_{\sigma_{1}}\circ\cdots\circ\phi^{l}_{\sigma_{l}}(0,1)
=I_{\sigma_1,\dots,\sigma_l}$.
    
    \item Let $i\leq k-1$ be the minimal index for which $\sigma_i\neq\sigma'_i$, and assume without loss of generality that $\sigma_i=0,\sigma'_i=1$. If $x\in I_{\sigma_1,\dots,\sigma_k}$, then
    \begin{align*}
    x
    &\leq
    \sup \phi^{1}_{\sigma_{1}}\circ\cdots\circ\phi^{i-1}_{\sigma_{i-1}}\circ\phi^{i}_{0}
    \circ\cdots\circ\phi^{k}_{\sigma_{k}}(0,1)
    \\
    &\leq
    \sup \phi^{1}_{\sigma_{1}}\circ\cdots\circ\phi^{i-1}_{\sigma_{i-1}}\circ\phi^{i}_{0}(0,1)
    \\
    &<\phi^{1}_{\sigma_{1}}\circ\cdots\circ\phi^{i-1}_{\sigma_{i-1}}(\frac{1}{2})
    \\
    &<\inf \phi^{1}_{\sigma_{1}}\circ\cdots\circ\phi^{i-1}_{\sigma_{i-1}}\circ\phi^{i}_{1}(0,1)
    \\
    &=\inf \phi^{1}_{\sigma'_{1}}\circ\cdots\circ\phi^{i-1}_{\sigma'_{i-1}}\circ\phi^{i}_{\sigma'_i}(0,1)
    \\
    &\leq \inf \phi^{1}_{\sigma'_{1}}\circ\cdots\circ\phi^{i-1}_{\sigma'_{i-1}}\circ\phi^{i}_{\sigma'_i}
    \circ\cdots\circ\phi^{k}_{\sigma'_{k}}(0,1)
    \\
    &=\inf I_{\sigma'_1,\dots,\sigma'_k}
    ~,
    \end{align*}
hence $x\notin I_{\sigma'_1,\dots,\sigma'_k}$.
\end{enumerate}
\end{proof}

With the parameterized family of segments in hand, we are ready to define a corresponding family of functions.
We define the function $h^{i}_{1}:[0,1]\setminus I^i_{1}\to\reals$ as follows
\[
h^{i}_{1}(x)
:=\begin{cases}
-\left(\frac{8^{i+1}-2}{8^{i+1}+2}\right)x+1\,, & 0\leq x\leq\frac{1}{2}+\frac{1}{8^{i+1}}\\
\left(\frac{8^{i+1}-2}{8^{i+1}-4}\right)x-\frac{2}{8^{i+1}-4}\,, & \frac{1}{2}+\frac{2}{8^{i+1}}\leq x\leq1
\end{cases}~,
\]
which is the piecewise linear function which satisfies $h(0)=h(1)=1,~h(\frac{1}{2}+\frac{1}{8^{i+1}})=h(\frac{1}{2}+\frac{2}{8^{i+1}})=\frac{1}{2}+\frac{1}{8^{i+1}}$, and also define $h^{i}_{0}:[0,1]\setminus I^i_{0}\to\reals$ as $h^{i}_{0}(x):=h^{i}_{1}(1-x)$ - see 
\figref{fig: h_1dim} for an illustration.
\begin{figure}
    \centering
    \includegraphics[scale=0.3, trim=0 50 0 150]{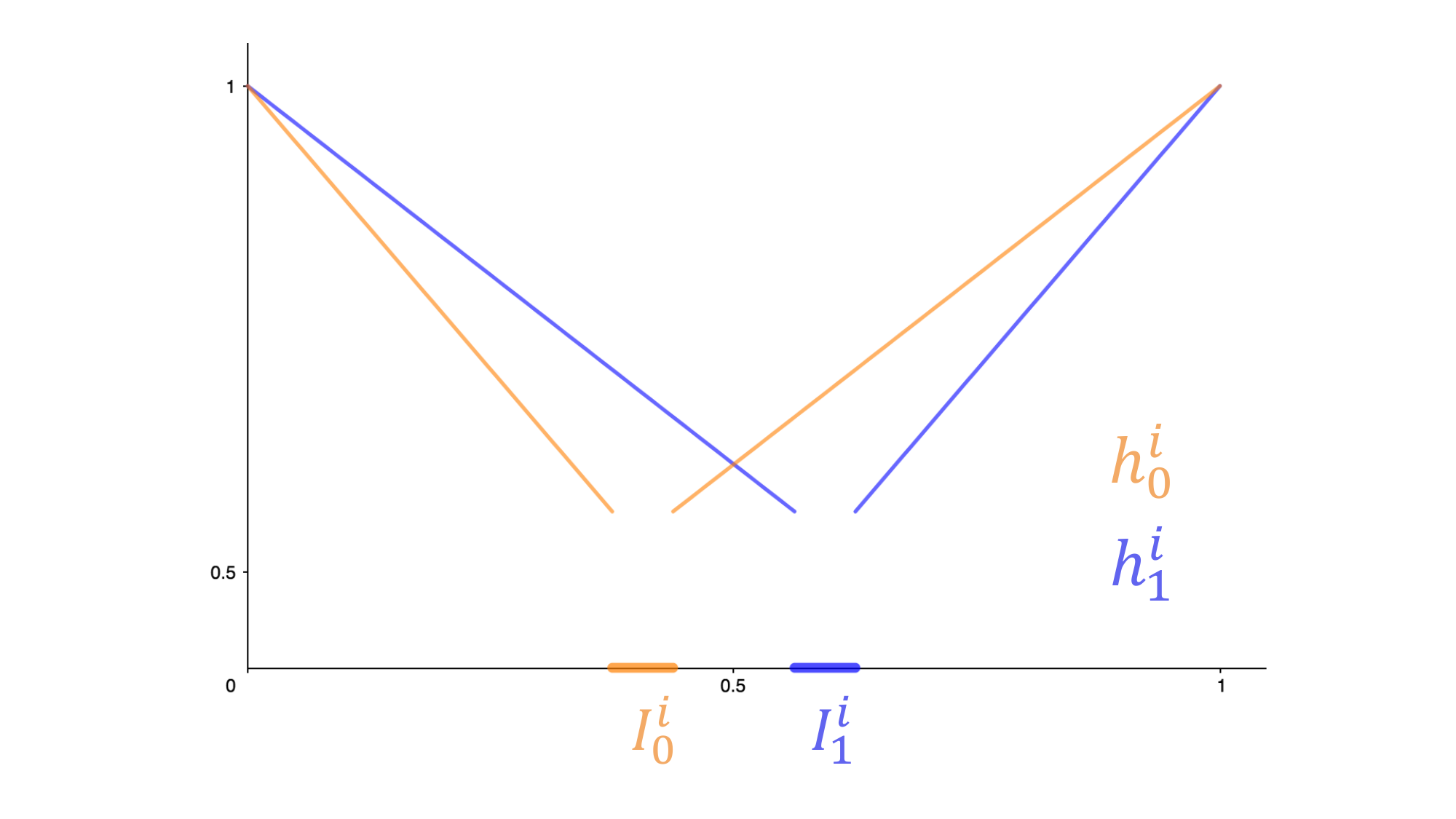}
    \caption{Illustration of $I_0,I_1$ and $h_0,h_1$.}
    \label{fig: h_1dim}
\end{figure}
We also denote $\Phi^i(x):=\phi^i_1(x)-\frac{1}{8^{i+1}}$ which is the affine function which satisfies $\Phi^i(0)=\frac{1}{2},\,\Phi^i(1)=\frac{1}{2}+\frac{1}{8^{i+1}},\,(\Phi^i)'=(\phi^i_0)'=(\phi^i_1)'$.

Finally, given any $N\geq2$ and $\sigma=(\sigma_1,\dots,\sigma_N)\in\{0,1\}^{N}$ we are ready to define $f:\reals\to\reals$ as
\[
f^{N}_{\sigma}(x)
:=
\begin{cases}
h^1_{\sigma_1}(x)
& x\in [0,1]\setminus I_{\sigma_1}
\\
\Phi^1\circ h^2_{\sigma_2}\circ(\phi^1_{\sigma_1})^{-1}(x)
& x\in I_{\sigma_1}\setminus I_{\sigma_1,\sigma_2}
\\
\Phi^1\circ\Phi^2\circ h^3_{\sigma_3} \circ (\phi^1_{\sigma_1}\circ\phi^2_{\sigma_2})^{-1}(x)
& x\in I_{\sigma_1,\sigma_2}
\setminus I_{\sigma_1,\sigma_2,\sigma_3}
\\
\vdots &
\\
\Phi^{1}\circ\cdots\circ\Phi^{N-1}\circ h^N_{\sigma_N} \circ(\phi^{1}_{\sigma_{1}}\circ\cdots\circ\phi^{N-1}_{\sigma_{N-1}})^{-1} (x)
& x\in 
I_{\sigma_1,\dots,\sigma_{N-1}}
\setminus I_{\sigma_1,\dots,\sigma_{N}}
\\
\Phi^{1}\circ\cdots\circ\Phi^{N-1}(0)
& x\in I_{\sigma_1,\dots,\sigma_{N}}
\\
1-x & x<0
\\
x & x>1
\end{cases}
~~.
\]
The following lemma show that the constructed function satisfies our desired requirements.

\begin{lemma}
For any $N\geq2,\,\sigma\in\{0,1\}^N:$
\begin{itemize}
    \item $f^N_\sigma$ is $2$-Lipschitz and convex.
    \item $f^N_{\sigma}(x)$ attains its minimum, and $\forall x\notin\arg\min_x f^N_\sigma(x)~\forall g\in\partial f^N_\sigma(x):|g|\geq \frac{1}{2}$.
    \item $f^N_{\sigma}(0)-\min_x f^N_{\sigma}(x)\leq 1$ and $\forall x\in\arg\min_x f^N_\sigma(x):|x|\leq1$.
\end{itemize}

\end{lemma}

\begin{proof}
Let $N\geq2,\,\sigma\in\{0,1\}^N$. We start by proving that $f^N_\sigma$ is continuous.
It is clear that $f^N_\sigma$ is piecewise linear since $h^i_{\sigma_i},\phi^i_{\sigma_i},\Phi^i$ are, So in order to show continuity we need to show continuity at the endpoints of adjacent linear pieces. Since by \lemref{lemma: I properties} it holds that $[0,1]\supset I_{\sigma_1}\supset I_{\sigma_1,\sigma_2}\supset\dots\supset I_{\sigma_1,\dots,\sigma_N}$, inspecting the definition of $f^N_\sigma$ we see that we need to verify continuity only at the endpoints of $I_{\sigma_1,\dots,\sigma_i}$ for any $i\in[N]$. 
For any $i\in[N]$, the left endpoint of $I_{\sigma_1,\dots,\sigma_i}$ satisfies
\begin{align*}
&\lim_{x\to (\inf I_{\sigma_1,\dots,\sigma_i})^+}f^N_\sigma(x)
\\
&=\lim_{x\to (\inf \phi^{1}_{\sigma_{1}}\circ\cdots\circ\phi^{i}_{\sigma_{i}}(0,1))^+}
\Phi^{1}\circ\cdots\circ\Phi^{i}\circ h^{i+1}_{\sigma_{i+1}} \circ(\phi^{1}_{\sigma_{1}}\circ\cdots\circ\phi^{i}_{\sigma_{i}})^{-1} (x) 
\\
&=\lim_{x\to0^+}
\Phi^{1}\circ\cdots\circ\Phi^{i}\circ h^{i+1}_{\sigma_{i+1}} (x) 
\\
&=
\Phi^{1}\circ\cdots\circ\Phi^{i}(\lim_{x\to0^+} h^{i+1}_{\sigma_{i+1}} (x))
\\
&=
\Phi^{1}\circ\cdots\circ\Phi^{i-1}\circ\Phi^{i}(1)
\\
&=
\Phi^{1}\circ\cdots\circ\Phi^{i-1}\circ\Phi^{i}\circ h^i_{\sigma_i}(0)
\\
&=\lim_{x\to (\inf \phi^{1}_{\sigma_{1}}\circ\cdots\circ\phi^{i-1}_{\sigma_{i-1}}(0,1))^-}
\Phi^{1}\circ\cdots\circ\Phi^{i-1}\circ h^{i}_{\sigma_{i}} \circ(\phi^{1}_{\sigma_{1}}\circ\cdots\circ\phi^{i-1}_{\sigma_{i-1}})^{-1} (x)
\\
&=\lim_{x\to (\inf I_{\sigma_1,\dots,\sigma_i})^-}f^N_\sigma(x)
~.
\end{align*}
The right endpoint follows a similar calculation, resulting in continuity.

Having established that $f^N_\sigma$ is piecewise linear and continuous, Lipschitzness will follow from the fact the each linear segment has a bounded slope. To see this holds, recall that
\begin{equation} \label{eq:Phi',phi' cancel}
(\Phi^i)'=(\phi^i)'\implies
(\Phi^i)'=\frac{1}{((\phi^i)^{-1})'}
~,
\end{equation}
so we get
\[
\left|\left(\Phi^{1}\circ\cdots\circ\Phi^{i-1}\circ h^i_{\sigma_i} \circ(\phi^{1}_{\sigma_{1}}\circ\cdots\circ\phi^{i-1}_{\sigma_{i-1}})^{-1} \right)'\right|
=\left|(h^i_{\sigma_i})'\right|
\leq\frac{8^{i+1}-2}{8^{i+1}-4}
=1+\frac{2}{8^{i+1}-4}<2~.
\]
In order to prove convexity, it is enough to show that each linear segment has slope which is larger than the segment to its left (hence the derivative is increasing which implies convexity).
Recalling that by \lemref{lemma: I properties} we have 
$[0,1]\supset I_{\sigma_1}\supset I_{\sigma_1,\sigma_2}\supset\dots\supset I_{\sigma_1,\dots,\sigma_N}$, we see that it is enough to establish monotonicity between the left linear segments of $I_{\sigma_1,\dots,\sigma_{i}},\,I_{\sigma_1,\dots,\sigma_{i+1}}$ (in that order) and between the right linear segments of $I_{\sigma_1,\dots,\sigma_{i+1}},\,,
I_{\sigma_1,\dots,\sigma_{i}}$ (in that order) for any $i\in[N]$.
In order to do that, first note that by definition of $h^i_1,h^i_0$ we have
\begin{align}\label{eq: h_i'}
(h_i^0)'(x)
&=
\begin{cases}
-\left(\frac{8^{i+1}-2}{8^{i+1}-4}\right)\,, & 0\leq x\leq\frac{1}{2}-\frac{2}{8^{i+1}}\\ \left(\frac{8^{i+1}-2}{8^{i+1}+2}\right)\,,
& \frac{1}{2}-\frac{1}{8^{i+1}}\leq x\leq1 \nonumber
\end{cases}
~~,
\\
(h^i_1)'(x)
&=
\begin{cases}
-\left(\frac{8^{i+1}-2}{8^{i+1}+2}\right)\,, & 0\leq x\leq\frac{1}{2}+\frac{1}{8^{i+1}}\\
\left(\frac{8^{i+1}-2}{8^{i+1}-4}\right)\,, & \frac{1}{2}+\frac{2}{8^{i+1}}\leq x\leq1
\end{cases}
~~.
\end{align}
Denoting by $L(I_{\sigma_1,\dots,\sigma_{i}})$ the left linear segment of $I_{\sigma_1,\dots,\sigma_{i}}$, we see that for any $x\in L(I_{\sigma_1,\dots,\sigma_{i}})$ we have by \eqref{eq:Phi',phi' cancel} and \eqref{eq: h_i'}:
\[
(f^N_{\sigma})'(x)
=\left(\Phi^{1}\circ\cdots\circ\Phi^{i}\circ h^{i+1}_{\sigma_{i+1}} \circ(\phi^{1}_{\sigma_{1}}\circ\cdots\circ\phi^{i}_{\sigma_{i}})^{-1}\right)'(x)
=\left(h^{i+1}_{\sigma_{i+1}}\right)'(x\in L([0,1]))
\leq -\left(\frac{8^{i+2}-2}{8^{i+2}+2}\right)~,
\]
while for any $x\in L(I_{\sigma_1,\dots,\sigma_{i+1}})$ we have
\[
(f^N_{\sigma})'(x)
=\left(\Phi^{1}\circ\cdots\circ\Phi^{i+1}\circ h^{i+2}_{\sigma_{i+2}} \circ(\phi^{1}_{\sigma_{1}}\circ\cdots\circ\phi^{i+1}_{\sigma_{i+1}})^{-1}\right)'(x)
=\left(h^{i+2}_{\sigma_{i+2}}\right)'(x\in L[0,1])
\geq-\left(\frac{8^{i+3}-2}{8^{i+3}-4}\right)
~.
\]
It is easy to verify for all $i\in\NN$ that $-\left(\frac{8^{i+3}-2}{8^{i+3}-4}\right)
\leq
-\left(\frac{8^{i+2}-2}{8^{i+2}+2}\right)$, which established the desired monotonicity.
Checking the right linear segments results in the same calculation, overall proving the first bullet.

For the second bullet, note that $f^N_\sigma$ attains its minimum at $I_{\sigma_1,\dots,\sigma_N}$ since till then the function is decreasing, and after that it's increasing (and on that segment it is constant). Elsewhere we already saw that the derivatives are $\pm\left(\frac{8^{i+1}-2}{8^{i+1}+2}\right),\pm\left(\frac{8^{i+1}-2}{8^{i+1}+4}\right)$ which are easily verified to have absolute value larger than $\frac{1}{2}$ for any $i\in\NN$.

The third bullet easily follows since $I_{\sigma_1,\dots,\sigma_N}\subset [0,1]$ and
$f^N_\sigma(0)=1,\,\Phi^1\circ\cdots\circ\Phi^{N-1}(0)> 0$.
\end{proof}

The following lemma follows immediately from the definition of $f^N_\sigma$ and from \lemref{lemma: I properties}, though we state it for future reference.

\begin{lemma} \label{lemma: I does not reveal}
For any $N\geq2,\,\sigma\in\{0,1\}^N,\,1\leq k<N$, it holds that $f^{N}_{\sigma}(x),(f^{N}_{\sigma})'(x)$ do not depend on $\sigma_{k+1},\dots,\sigma_{N}$ for $x\notin I_{\sigma_1,\dots,\sigma_k}$.
\end{lemma}

Given \lemref{lemma: I properties} and \lemref{lemma: I does not reveal}, we state the following lemma which whose proof appears in \citep[Proof of Proposition 1]{kornowski2021oracle}.

\begin{lemma}[\citep{kornowski2021oracle}] \label{lemma: prob jump into segment}
For any $t\in\NN$, $1\leq l<k\leq N:$ 
\[\Pr_{\sigma}\left[
x_{t+1}\in I_{\sigma_{1},\dots,\sigma_{l},\dots,\sigma_{k}}
\middle|x_1,\dots,x_{t}\notin I_{\sigma_{1},\dots,\sigma_{l}}
\right]
\leq \frac{1}{2^{k-l-1}}
~.
\]
\end{lemma}

We now claim that points that are $\delta$ close to $\arg\min_x f^N_\sigma(x)=I_{\sigma_1,\dots,\sigma_N}\subset I_{\sigma_1,\dots,\sigma_{N-1}}\subset I_{\sigma_1,\dots,\sigma_{N-2}}\subset\dots$ are necessarily inside $I_{\sigma_1,\dots,\sigma_k}$ with large enough $k$.

\begin{lemma} \label{lemma: delta distance segment}
Let $k=\lfloor\frac{1}{4}\sqrt{\log(1/\delta)}\rfloor<N\in\NN$. Then $\dist(x,I_{\sigma_1,\dots,\sigma_N})<\delta$ implies $x\in I_{\sigma_1,\dots,\sigma_{k}}$.
\end{lemma}

\begin{proof}
Recall that $I_{\sigma_1,\dots,\sigma_{N}}\subset I_{\sigma_1,\dots,\sigma_{k}}$, and note that for any $i\in\NN:$
\begin{align*}
    \inf I_{\sigma_1,\dots,\sigma_{i+1}}-\inf I_{\sigma_1,\dots,\sigma_{i}}
    &=\phi^{1}_{\sigma_{1}}\circ\cdots\circ\phi^{i}_{\sigma_i}\circ\phi^{i+1}_{\sigma_{i+1}}(0)
    - \phi^{1}_{\sigma_{1}}\circ\cdots\circ\phi^{i}_{\sigma_{i}}(0)
    \\
    &\overset{(*)}{=}\left[\prod_{j=1}^{i}(\phi^j_{\sigma_j})'\right]\cdot\left(\phi^{i+1}_{\sigma_{i+1}}(0)-0\right)
    \\
    &\geq \left[\prod_{j=1}^{i}\frac{2}{8^{j+1}}\right]\cdot \left(\frac{1}{2}-\frac{2}{8^{i+2}}\right)
    \\
    &\geq \frac{2^{i}}{8^{\frac{i(i+3)}{2}}}\cdot \frac{1}{4}~.
\end{align*}
where $(*)$ follows from the fact that for any affine mapping $x\mapsto ax+b$ it holds that $(ax_1+b)-(ax_2+b)=a(x_1-x_2)$. Thus
\begin{align*}
    \inf I_{\sigma_1,\dots,\sigma_{N}}-\inf I_{\sigma_1,\dots,\sigma_{k}}
    &=\sum_{i=k}^{N-1}\left(\inf I_{\sigma_1,\dots,\sigma_{i+1}}-\inf I_{\sigma_1,\dots,\sigma_{i}}\right)
    \\
    &\geq \sum_{i=k}^{N-1}\frac{2^{i}}{8^{\frac{i(i+3)}{2}}}\cdot \frac{1}{4}
    \\
    &>\frac{2^{k}}{4\cdot 8^{\frac{k(k+3)}{2}}}
    >\delta
    ~,
\end{align*}
where that last inequality is a straightforward computation for $k=\lfloor\frac{1}{4}\sqrt{\log(1/\delta)}\rfloor$. We conclude that \[
\inf I_{\sigma_1,\dots,\sigma_{k}}<\inf I_{\sigma_1,\dots,\sigma_{N}}-\delta
~,
\]
while the same arguments also show that $\sup I_{\sigma_1,\dots,\sigma_{k}}>\sup I_{\sigma_1,\dots,\sigma_{N}}+\delta$. Overall, we see that for any $x$ satisfying $\dist(x,I_{\sigma_1,\dots,\sigma_N})<\delta:$
\[
\inf I_{\sigma_1,\dots,\sigma_k}<I_{\sigma_1,\dots,\sigma_N}-\delta<x<\sup I_{\sigma_1,\dots,\sigma_N}+\delta<\sup I_{\sigma_1,\dots,\sigma_k}
\implies x\in I_{\sigma_1,\dots,\sigma_k}
~.
\]

\end{proof}

We are now ready to finish the proof. We set $k=\lfloor\frac{1}{4}\sqrt{\log(1/\delta)}\rfloor,\,N=k+1,\,\sigma\sim\mathrm{Unif}\{0,1\}^N$. We consider the algorithm's iterates $x_1,\dots,x_T$ as random variables which depend on the random choice of the function $f_{\sigma}$ fed to the algorithm.
Denote the stochastic process \[
Z_{0}=0,~Z_t:=\max\{l\in\NN:\exists s\leq t,x_s\in I_{\sigma_1,\dots,\sigma_l}\}
~,
\]
and note that that $Z_{t+1}-Z_{t}\geq0$ with probability 1, yet by \lemref{lemma: prob jump into segment}: $\Pr[Z_{t+1}-Z_{t}=m]\leq\frac{1}{2^{m-1}}$. Hence by \lemref{lemma: delta distance segment}
\begin{align*}
&\Pr_{\sigma}\left[\exists t\in[T]:
\dist(x_{t},\arg\min_x f^N_\sigma(x))<\delta\right]
\leq\Pr_{\sigma}\left[\exists t\in[T]:
x_{t}\in I_{\sigma_1,\dots,\sigma_k}\right]
=\Pr_{\sigma}\left[Z_{T}\geq k\right]
\\
&\leq\frac{1}{k}\E[Z_T]
=\frac{1}{k}\sum_{j=1}^{T}\E[Z_{j}-Z_{j-1}]
=\frac{1}{k}\sum_{j=1}^{T}\sum_{m=0}^{\infty}\Pr[Z_{j}-Z_{j-1}=m]
\\
&\leq\frac{1}{k}\sum_{j=1}^{T}\sum_{m=0}^{\infty}\frac{1}{2^{m-1}}
\leq\frac{1}{k}\sum_{j=1}^{T}\sum_{m=0}^{\infty}\frac{1}{2^{m-1}}
\leq\frac{1}{k}\sum_{j=1}^{T}4
=\frac{4T}{k}
<\frac{2}{3}
~.
\end{align*}
This proves \eqref{eq: Yao hard distribution} for 2-Lipschitz function and $\frac{1}{2}$-stationary points. By rescaling the Lipschitz constant and $\epsilon$ by a factor of 2, we have finished the proof.

\subsection{Proof of \thmref{thm: nonconvex 1dim}}

The proof is identical to the proof of \thmref{thm: convex 1dim} which appears above, by replacing the constructed family $f_\sigma$ by the family $h_{\sigma}$ constructed in \citep[Proof of proposition 1]{kornowski2021oracle}. The only modification throughout the whole proof is that \lemref{lemma: delta distance segment} holds with $k=\Theta(\log(1/\delta))$ instead of $k=\frac{1}{4}\sqrt{\log(1/\delta)}$, affecting the final bound on $T$ appropriately.

\subsection{Proof of \thmref{thm:delta,epsilon upper bound}}

It is well known (see for example \citep{bubeck2015convex}) that given a Lipschitz convex function, performing projected gradient descent for $T_1:=\frac{L^2 R^{2/3}}{\epsilon^2\delta^{2/3}}$ iterations and setting $x_{\GD}:=\frac{1}{T_1}\sum_{j=1}^{T_1}x_j$
satisfies 
$f(x_{\GD})-\inf_x f(x)\leq\frac{RL}{\sqrt{T_1}}=R^{2/3}\epsilon\delta^{1/3}$.
Furthermore, according to \citep[Theorem 8]{zhang2020complexity}, if we initialize $\INGD$ at a point $x$ such that $f(x)-\inf_x f(x)\leq \Delta$ then with probability $\frac{2}{3}$ it produces a $(\delta,\epsilon)$ stationary point within $T_2=\widetilde{O}\left(\frac{\Delta L^2}{\delta\epsilon^3}\right)$ oracle calls. In particular, initializing at $x_{\GD}$ for which we have $\Delta=R^{2/3}\epsilon\delta^{1/3}$ yields $T_2=\frac{L^2 R^{2/3}}{\epsilon^2\delta^{2/3}}$
and letting $T=T_1+T_2$ proves the claim. Replacing $\INGD$ with its deterministic counterpart from \thmref{thm: deterministic upper} proves the additional claim.

\subsection{Proof of \claimref{claim: 1dim_equivalence}}

First, if $x$ is $\delta$-close to an $\epsilon$-stationary point it is clear that it is also $(\delta,\epsilon)$-stationary.

For the nontrivial implication, suppose $x$ is $(\delta,\epsilon)$-stationary. Namely, there exists $g\in\conv(\cup_{x-\delta<t<x+\delta}\partial f(t))$ with $|g|\leq\epsilon$. By Caratheodory's theorem (for $d=1$), this means there exist $t_1,t_2\in(-\delta,\delta),g_i\in\partial f(t_i)$ and $\lambda_1,\lambda_2\in[0,1]$ satisfying $\lambda_1+\lambda_2=1$ such that $g=\lambda_1 g_1+\lambda_2 g_2$. If either $g_1$ or $g_2$ satisfy $|g_i|\leq\epsilon$ then we get that $x$ is $\delta$-close the $\epsilon$-stationary point $t_i$. Otherwise, both $g_i$ satisfy $|g_i|>\epsilon$, but their convex combination is of magnitude smaller than $\epsilon$ - thus they necessarily have different signs. Without loss of generality suppose $g_1<0,g_2>0$. We claim that in this case there exists $t^*\in[t_1,t_2]$ such that $0\in\partial f(t^*)$, finishing the proof since $x$ is $\delta$-close to the stationary point $t^*$. Indeed, let $A:=\{t\in[t_1,t_2]:\forall g\in\partial f(t),g<0\}$ and note that by assumption $t_1\in A, t_2\not\in A$. Denote $t^*:=\sup A\in[t_1,t_2]$ which exists since $A$ is non-empty. On the one hand, by definition of the supremum there exists a non-decreasing sequence $\{y_n\}_{n=1}^{\infty}\subset A$ with $\lim_{n\to\infty}y_n=t^*$. Note that by definition of $A$ all the sub-differentials at all of $y_n$ are negative, thus $t^*$ has some non-positive sub-differential (by Bolzano–Weierstrass one can construct a converging sub-sequence of derivatives). On the other hand, by construction all differentiable points in $(t^*,t_2)$ have only positive differentials, thus $x+t^*$ has some non-negative sub-differential. Finally, recalling that $\partial f(t^*)$ is convex we get $0\in\partial f(t^*)$ as claimed.

\subsection*{Acknowledgements}

We thank Xiaotong Yuan for pointing out a bug in a previous version of the proof of Lemma 1.
This research is supported in part by European Research Council (ERC) grant 754705, and a grant from the Council for Higher Education Competitive Program for Data Science Research Centers.

\bibliography{bib}
\bibliographystyle{plainnat}

\appendix

\section{Appendix}

\begin{lemma} \label{lemma: delta,epsilon not delta close}
For any $C>1,\,\delta<1$ and any $\epsilon<\frac{1}{8C}$, there exists a 1-Lipschitz convex function $f:\reals^d\to\reals$ in dimension $d=O(1/\epsilon^2)$ and a point $x\in\reals^d$ such that $x$ is a $(\delta,\epsilon)$-stationary point of $f$, yet is at distance at least $C\delta$ away from any $\epsilon$-stationary point of $f$.
\end{lemma}

\begin{proof}
Fix $\delta<1<C,~\epsilon<\frac{1}{8C}$, and let $d=\lceil\frac{3}{\epsilon^2}\rceil$. We denote by $A:=\mathrm{diag}(2\epsilon^2,1,\dots,1)\in\reals^{d\times d}$ the diagonal positive-definite matrix whose diagonal is $(2\epsilon^2,1,\dots,1)$, and consider its corresponding Mahalanobis norm $f(x):=\sqrt{x^\top Ax}$. For any $x\neq 0$ we have
\[
\nabla f(x)=\frac{1}{\sqrt{x^\top Ax}}\cdot Ax
~\implies~
\norm{\nabla f(x)}^2=\frac{x^\top A^2 x}{x^\top A x}
\in[\lambda_{\min}(A),\lambda_{\max}(A)]=[2\epsilon^2,1]~.
\]
In particular, we see that for any $x\neq0$ the gradient norm is at least $\sqrt{2\epsilon^2}=\sqrt{2}\cdot\epsilon$, thus it is not an $\epsilon$-stationary point.
We deduce that $f$ is indeed 1-Lipschitz, and that it has no $\epsilon$-stationary points except for the origin (which is its minimum, hence stationary). Denote $x_{0}:=\frac{\delta}{8\epsilon}\mathrm{e}_1$ and notice that its distance from the only $\epsilon$-stationary point of $f$ is $\norm{x_0-0}=\frac{\delta}{8\epsilon}>C\delta$. It remains to show that $x_0$ is a $(\delta,\epsilon)$-stationary point of $f$. To that end, consider for any $2\leq j\leq d:~z_j:=x_0+\frac{\delta}{2}\cdot\mathrm{e}_j=\frac{\delta}{8\epsilon}\mathrm{e}_1+\frac{\delta}{2}\cdot\mathrm{e}_j$ which are clearly inside a $\delta$-ball around $x_0$. Furthermore,
\[
\nabla f(z_j)
=\frac{1}{\sqrt{\frac{\delta^2}{32}+\frac{\delta^2}{4}}}\cdot\left(\frac{\delta\epsilon}{4}\cdot\mathrm{e}_1+\frac{\delta}{2}\cdot\mathrm{e}_j\right)
=\frac{\sqrt{2}}{3}\cdot\epsilon\cdot\mathrm{e}_1
+\frac{2\sqrt{2}}{3}\cdot\mathrm{e}_j~.
\]

Now, define $g:=\frac{1}{d-1}\sum_{j=2}^{d}\nabla f(z_j)$ and notice that by definition $g\in\partial f_\delta(x_0)$ since $\norm{x-z_j}<\delta$. Furthermore,
\begin{align*}
g&=\frac{1}{d-1}\sum_{j=2}^{d}\nabla f(z_j)
=\frac{\sqrt{2}}{3}\epsilon \cdot \mathrm{e}_1 + \frac{2\sqrt{2}}{3(d-1)}\sum_{j=2}^{d}\mathrm{e}_j
\\
\implies
\norm{g}&=\sqrt{\frac{2\epsilon^2}{9}+\frac{8}{9(d-1)^2}\cdot(d-1)}
=
\sqrt{\frac{2\epsilon^2}{9}+\frac{8}{9(d-1)}}<\epsilon~,
\end{align*}
where the last inequality is easily verified for $d=\lceil\frac{3}{\epsilon^2}\rceil$. This establishes that $x$ is indeed a $(\delta,\epsilon)$-stationary point.

\end{proof}

\end{document}